\documentclass[12pt,twoside,reqno]{amsart}

\usepackage[margin=1.4in]{geometry}
\usepackage[hidelinks]{hyperref}

\usepackage{color}
\usepackage{graphicx}
\usepackage{amsmath}
\usepackage{amssymb}
\usepackage{amsfonts}
\usepackage{multicol}
\usepackage[ansinew]{inputenc}
\usepackage{amscd}
\usepackage[mathscr]{eucal}
\usepackage{hyperref}

\newcommand{\R}{\mathbb{R}}

\newcommand{\N}{\mathbb{N}}

\newcommand{\Q}{\mathbb{Q}}

\newcommand{\SL}{{\rm SL}}

\newcommand{\GL}{{\rm GL}}

\newcommand{\Mat}{{\rm Mat}}

\newcommand{\Uscr}{\mathscr{U}}
\newcommand{\tops}{\mathscr{V}}

\newcommand{\abs}[1]{\bigl| #1 \bigr|} 
\newcommand{\norm}[1]{\lVert#1\rVert} 
\newcommand{\normtwo}[1]{
	{\left\vert\kern-0.25ex\left\vert\kern-0.25ex\left\vert #1
		\right\vert\kern-0.25ex\right\vert\kern-0.25ex\right\vert} }

\newcommand{\GLmR}{\GL(d, \R)}

\newcommand{\Gr}{{\rm Gr}}
\newcommand{\Pp}{\mathbb{P}}

\newcommand{\bld}[1]{\mathbf{{#1}}}

\newcommand\restr[2]{{
		\left.\kern-\nulldelimiterspace 
		#1 
		\vphantom{\big|} 
		\right|_{#2} 
}}

\theoremstyle{plain}
\newtheorem{theorem}{Theorem}[section]
\newtheorem{proposition}{Proposition}[section]
\newtheorem{corollary}[proposition]{Corollary}
\newtheorem{lemma}[proposition]{Lemma}

\theoremstyle{definition}
\newtheorem{definition}{Definition}[section]

\theoremstyle{definition}
\newtheorem{remark}{Remark}[section]

\numberwithin{equation}{section}

\newcommand{\Prob}{\mathrm{Prob}}

\newcommand{\rank}{\mathrm{rank}}
\newcommand{\Ker}{\mathrm{K}}

\title [non-invertible random cocycles with constant rank]{Continuity of the Lyapunov exponents of non-invertible random cocycles with constant rank}

\date{}

\begin{document}

\author[P. Duarte]{Pedro Duarte}
\address{Departamento de Matem\'atica and CMAFcIO\\
Faculdade de Ci\^encias\\
Universidade de Lisboa\\
Portugal
}
\email{pmduarte@fc.ul.pt}

\author[C. Freijo]{Catalina Freijo}
\address{Instituto de Matem\'atica e Estatistica\\
	Universidade de S\~ao Paulo\\
	Brasil
}
\email{catalinafreijo@ime.usp.br}

\begin{abstract}
In this paper we establish uniform large deviations estimates of exponential type and H\"older continuity of the Lyapunov exponents for random non-invertible cocycles with constant rank.
\end{abstract}

\maketitle

\tableofcontents

\section{Introduction}
\label{intro}

In Ergodic Theory the Lyapunov exponent of an i.i.d. matrix valued random process $A_1,  A_2, \ldots $ measures the growth
rate of norms of products of matrices in the process
\begin{equation}
\label{lexp}
 L=\lim_{n\to \infty} \frac{1}{n}\, \log \norm{A_n\, \cdots \, A_1} . 
\end{equation}
The almost sure existence of this limit was established by
H. Furstenberg and H. Kesten~\cite[1960]{FK60} for ergodic processes, i.e., processes generated by a matrix valued function over the dynamics of an ergodic measure preserving transformation. The   i.i.d. processes above are generated by  locally constant matrix valued functions on spaces of sequences and Bernoulli shifts on those spaces.

This result  marked the beginning of the Theory of Lyapunov exponents of linear cocycles, which was later greatly amplified by the 
Multiplicative ergodic theorem~\cite[1968]{Oseledet68} of V. I. Oseledets. This theorem provides a global description of all (directional) Lyapunov exponents\footnote{The Lyapunov exponent in~\eqref{lexp} is sometimes called the top or first Lyapunov exponent because it is the maximum of all  directional Lyapunov exponents. }  of a linear cocycle.
An important goal of this theory is  understanding the stability, or continuity, of the Lyapunov exponents  in terms of the data defining a linear cocycle.

Next we outline some of the main achievements of this theory in the last 60 years  regarding the previous goal,
for the special class of (invertible) random linear cocycles, i.e., generated by a locally constant matrix valued function $A:X\to \GL(d,\R)$  over a Bernoulli or Markov shift $\sigma:X\to X$. As usual, $\GL(d,\R)$ denotes the general linear group of $d\times d$ invertible matrices with real entries.

 A first milestone of this theory was Furstenberg's formula in~\cite[1963]{Fur63} which expresses the top Lyapunov exponent
 via an integral formula involving a so-called \textit{stationary measure} of the random linear cocycle.\footnote{The definition and existence of stationary measures for non-invertible linear cocycles, i.e., determined by  locally constant functions $A:X\to\Mat(d,\R)$, is problematic. This explains why a theory of (non-invertible) random linear cocycles has been undeveloped. }
In~\cite[1983]{FK83}  H. Furstenberg and Y. Kifer used this formula to
prove the continuity of the (top)
 Lyapunov exponent under a generic irreducibility  assumption on the random linear cocycle.
 
A few years later  E. Le Page~\cite[1989]{LePage} 
proved that the top Lyapunov exponent is actually  H\"older continuous if together with the irreducibility assumption  we assume  the simplicity of the first Lyapunov exponent. The proof is based on a spectral method, which exploits the existence of a   gap in the spectrum of a certain Markov operator\footnote{The mentioned spectral gap  refers the spectrum of the Markov operator being contained in some closed disk of radius $<1$, with the exception of the simple eigenvalue $1$ associated with the constant functions (the operator's fixed points).} that traces the action of the linear cocycle on the projective space. The existence of a spectral gap is forced by the generic hypothesis of Le Page's continuity theorem.
This spectral method had been used earlier by Le Page~\cite[1982]{LP82} to establish limit theorems such as a large deviation principle and a central limit theorem for random linear cocycles over Bernoulli shifts. Similar limit theorems for random cocycles over Markov shifts were obtained by P. Bougerol~\cite[1988]{Bou}.

Regarding the problem of general continuity of the Lyapunov exponents (without generic hypothesis),   C. Bocker-Neto and M. Viana~\cite[2016]{BV} proved that this continuity always holds for random  $\GL(2,\R)$-valued linear cocycles. A similar result was announced  a few years ago by A. Avila, A. Eskin and M. Viana to hold for random $\GL(d,\R)$-valued  cocycles (any $d\geq 2$). See Section 10.7 in M. Viana's book~\cite[2014]{LLE}.
In the same direction, E. Malheiro and M. Viana~\cite[2015]{MV2015}   proved the continuity of the Laypunov exponents for random $\GL(2,\R)$-valued  linear cocycles over mixing Markov shifts. Still in the framework of $\GL(2,\R)$-valued linear cocycles, these results where  extended beyond the class of random cocycles by
L. Bakes, A. Brown and C.Butler~\cite[2018]{BBB} who proved a conjecture by M. Viana on the general continuity of the Lyapunov exponents for fiber-bunched cocycles over   hyperbolic base dynamics.

Concerning regularity, the best we can hope for is that the Lyapunov exponent  is an analytic function of a parameter, given a family  of linear cocycles that depends analytically on that parameter. Analyticity of the top Lyapunov exponent is known to hold in two special cases: First by a theorem of D. Ruelle~\cite[1982]{Ru82}, for uniformly hyperbolic
$\SL(2,\R)$-cocycles and more generally for $\GL(d,\R)$-cocycles with  dominated splitting and a one-dimensional strongly unstable direction.
Second by a theorem Y. Peres ~\cite[1991]{Pe91}, for random cocycles generated by  locally constant $\GL(d,\R)$-valued functions with finitely many values, where is shown that any Lyapunov exponent, if simple,  is locally analytic as a function of the probability weights.

  An example of a random $\SL(2,\R)$-cocycle by B. Haperin, see ~\cite[Appendix A]{ST85},
  shows that the H\"older modulus of continuity of Le Page's~\cite[Theorem 1]{LP89} is  optimal. The recent  preprints~\cite[2022]{BeDu22, BCDFK22} deepen  the fact illustrated by this example, showing that if the ratio between the metric entropy of the shift and the Lyapunov exponent is less than $1$ there exists a  dichotomy on the regularity of the Lyapunov exponent, which can either be analytic when the cocycle is uniformly hyperbolic, or else H\"older continuous with a H\"older exponent which can not exceed the said ratio when the cocycle is not uniformly hyperbolic.
  
  It is natural to ask about the modulus of continuity of
  the Lyapunov exponent  as a function of a  random cocycle which does not satisfy the assumptions of Le Page~\cite[Theorem 1]{LP89}. Recently E. H. Tall and M. Viana~\cite[2020]{Tall-Viana} proved that for any random $\GL(2,\R)$-cocycle, the Lyapunov exponents  are always at least pointwisely $\log$-H\"older continuous\footnote{Given a  metric space $(X,d)$, a function $f:X\to\R$  is said to be 
  pointwisely $\log$-H\"older continuous at $a\in X$ if there exists  $C<\infty$ such that
  $| f(x)-f(a)|\leq C\,\left(  \log \frac{1}{d(x,a)} \right)^{-1}$ for all $x$ in some neighborhood of $a$.
}
  and moreover pointwisely H\"older continuous\footnote{A function $f:X\to\R$   is said to be 
  	pointwisely  H\"older continuous at $a\in X$ if there are  $C<\infty$ and $0<\alpha<1$ such that
  	$| f(x)-f(a)|\leq C\, d(x,a)^{\alpha}$ for all $x$ in some neighborhood of $a$.} when the Lyapunov exponents are simple. In basically the same setting,   random $\GL(2,\R)$-cocycles with finitely many values, P. Duarte and S. Klein ~\cite[2020]{DK-Holder} proved that the Lyapunov exponents, if simple, are locally weak-H\"older continuous\footnote{A function $f:X\to\R$   is said to be 
  	locally weak-H\"older continuous at $a\in X$ if there are  $C<\infty$, $c>0$ and $0<\alpha<1$ such that
  	$| f(x)-f(y)|\leq C\,e^{-c\, \left(\log 1/d(x,y)\right) ^{\alpha} }$ for all $x, y$ in some neighborhood of $a$.
  When $\alpha=1$, weak-H\"older becomes H\"older continuity with $c$-exponent.}.
  In ~\cite[2019]{DKS19} P. Duarte, S. Klein and M. Santos
  provided an example of a random $\SL(2,\R)$-cocycle  where the assumptions of~\cite[Theorem 1]{LP89} fail to hold and the regularity of the Lyapunov exponent is neither H\"older nor weak-H\"older. In fact it can not be better than the very bad $\log^3$-H\"older modulus of continuity.

In~\cite[Theorem 3.1]{DK-book}, P. Duarte and S. Klein  
proved that the existence of locally uniform large deviation estimates of exponential type for a linear cocycle (definition in Subsection~\ref{Section: LDT estimates}) is sufficient for the H\"older continuity of the Lyapunov exponent. This result  is used later in the proof of the main theorem. As an application, it was proved in ~\cite[Chapter 5]{DK-book} an analogue of Le Page's H\"older continuity theorem for random $\GL(d,\R)$-cocycles over mixing Markov shifts.

With the single exception of Furstenberg-Kifer~\cite[Theorem B]{FK83}, all results described so far assume that the random cocycles are generated by a distribution law in $\GL(d,\R)$ which is compactly supported. 
Recently in~\cite[2020]{SV20}, A. S\'anchez and M. Viana 
proved that the top Lyapunov exponent is upper semi-continuous, but not continuous, 
with respect to the Wasserstein distance in the class of 
random and non-compactly supported $\SL(2,\R)$-cocycles.

\bigskip

Let us now turn to (fiber-wise) non-invertible linear cocycles. Many aspects of the theory of linear cocycles have been extended to such cocycles, namely in what concerns the Multiplicative ergodic theorem: Blumenthal-Young \cite{BY17},
P. Thieulle \cite{T87}, Backes-Maur\'icio~\cite{BP15}.
In~\cite[2015]{FGTQ2015}, G. Froyland, C. Gonz\'alez-Tukman and A. Quas have proved the stability of the  Lyapunov exponents and Oseledets subspaces
under small perturbations with  uniform noise. Notice however that stability is weaker than continuity, which also allows for non-random perturbations.

There are many obstacles to develop a theory for the continuity of the Lyapunov exponents of random  $\Mat(d,\R)$-cocycles.
First there is no  Furstenberg's theory about the relation between the first Lyapunov exponent and the stationary measures of the random cocycle. Via exterior algebra we can always reduce the study of all Lyapunov exponents to the first one.
The problem starts with the fact that
non-invertible matrices only induce a partial action on the projective space. For instance the action of the zero matrix is nowhere defined. Because of this,  stationary measures may not exist,
which explains why such a general theory is impossible.
Discontinuities as in ~\cite{SV20} can be carried over to this framework. For instance, given a random $\GL(2,\R)$-cocycle with unbounded support, but  co-norm bounded away from $0$, which happens to be a discontinuity point of the first Lyapunov exponent, then the inverse cocycle becomes a random and compactly supported $\Mat(2,\R)$-valued cocycle where the second Lyapunov exponent is discontinuous. Another issue is the existence of \textit{eventually vanishing cocycles}, i.e., cocycles whose iterates eventually vanish on some set with positive measure, which by ergodicity must have first Lyapunov exponent equal to $-\infty$. We have a problem if
the set of eventually vanishing cocycles is dense in some open set.
The problem is even greater if such cocycles can be approximated by other random cocycles whose first Lyapunov exponent is bounded away from $-\infty$. In this case, the top Lyapunov exponent becomes nowhere continuous. Such an example is provided in Section~\ref{example}.

In order to develop a positive theory we need to consider classes of random non-invertible cocycles where eventually vanishing cocycles form a closed nowhere dense set.
 One such possibility, the one pursued in this manuscript, is to consider the class of random  cocycles with  finitely many values of constant rank. The strategy of analysis is a simple reduction to the invertible case.

\bigskip

The paper is organized as follows:
In section 2 we introduce the needed concepts and state the main results.
Section 3 is dedicated to an example of a random $\Mat(2,\R)$-valued family of cocycles where the top Lyapunov exponent is nowhere continuous,
which illustrates the limitations to the development of such sort of theory.
Section 4 explains the core reduction of random constant rank cocycles to invertible random cocycles.
In section 5 we prove the main results.

\section{Statement of results}
\label{statement}

\subsection{General cocycles}

Let  $\Mat(d,\R)$ be the algebra of square $d$-matrices.

Consider a measure preserving dynamical system
$(X,\mu,T)$, where $(X,\mu)$ is a probability space and $T:X\to X$ a measurable transformation  with $T_\ast\mu=\mu$. 
A  skew-product map
$F:X\times \R^d\to X\times\R^d$,
$F(x,v):=(T x, A(x)\, v)$, 
where $A:X\to\Mat(d,\R)$ is a measurable function
is called a   \textit{linear cocycle}.
We refer to both $F$ and the pair $(T,A)$ as the linear cocycle. Its iterates   are given by
$F^n(x, v):=(T^n x,  A^n(x)\, v)$, where 
$$ A^n(x):=  A(T^{n-1} x)\, \cdots\, A(T x)\,  A(x).$$

A linear cocycle $(T,A)$ is called \textit{integrable} if
$\log \norm{A}\in L^1(X,\mu)$, i.e.,
$$ \int_X |\log \norm{A(x)}|\, d\mu(x) < \infty . $$

By Furstenberg-Kesten~\cite{FK60}, if the cocycle is integrable then the following limit exists for $\mu$-almost every $x\in X$ to a $T$-invariant function
$$ L_1(T,A; x):=\lim_{n\to\infty}
\frac{1}{n}\, \log \norm{A^n(x)} .$$
If moreover $(X,\mu,T)$ is ergodic then this limit
function is constant $\mu$-almost surely and its limit value $L_1(T,A)$ is called the \textit{first Lyapunov exponent} of $(T,A)$.
Under the ergodic assumption, $L_1(T,A)$ is still well defined as the same $\mu$-almost sure limit,  taking values $\pm\infty$, when the cocycle is not integrable.

Denote by $s_1(A)\geq s_2(A)\geq \cdots \geq s_d(A)\geq 0$ the \textit{singular values} of a matrix $A\in\Mat(d,\R)$. 
Given $1\leq k\leq d$, the \textit{$k$-th exterior power} of a matrix $A\in\Mat(d,\R)$ is denoted by $\wedge_k A$. This is the matrix  $\wedge_k A:=( a_{I,J})_{I,J}$ indexed over all subsets $I\subset \{1,\ldots, d\}$ with $k$ elements, where $a_{I,J}:=\det(A_{I,J})$ and $A_{I,J}$ denotes the submatrix of $A$ indexed in $I\times J$. Two important properties of this construction are (see~\cite{St83}):
\begin{enumerate}
	\item $\wedge_k(A\,B)=(\wedge_k A)\, (\wedge_k B)$,\; $\forall\, A,B\in\Mat(d,\R)$;
	\item $\norm{\wedge_k A}=s_1(A)\, \cdots\, s_k(A)$,\; $\forall\, A\in\Mat(d,\R)$.
\end{enumerate}

Given a cocycle $(T,A)$,
its $k$-th exterior power cocycle is $(T, \wedge_k A)$
where $\wedge_k A:X\to \Mat( \binom{d}{k},\R)$ is the function  $(\wedge_k A)(x):=\wedge_k A(x)$. From property (1) above it follows that for all $x\in X$ and $n\in\N$,
$$ (\wedge_k A)^n(x) = \wedge_k A^n(x) .$$
Notice that the integrability of $A$ does not imply that of $\wedge_k A$, but we always have
$\log \norm{\wedge_k A}\leq k\, \log \norm{A}$, which implies that $L_1(\wedge_k A)<k\, L_1(A)<\infty$.

In general, when $(X,\mu,T)$ is ergodic and $1\leq i\leq d$ we define the $i$-th Lyapunov exponent as the $\mu$-almost sure limit
$$ L_i(T, A):=\lim_{n\to\infty}
\frac{1}{n}\, \log s_i(A^n(x)) .$$
The $\mu$-almost sure existence of this limit follows from applying Furstenberg-Kesten's theorem to the exterior power cocycles $\wedge_j A$ with $1\leq j\leq k$. This also implies the relation
$$ L_1(\wedge_k A) = L_1(A) +\cdots + L_k(A) .$$


\subsection{Random cocycles}
Consider the $(m-1)$-dimensional  simplex
$$ \Delta^{m-1}:=\left\{(p_1,\ldots, p_m)\in \R^m \, \colon\,  \sum_{j=1}^m p_j=1,\, \text{ with }\;  p_j\geq 0 \quad  \forall j=1,\ldots, m \, \right\}$$
consisting of all probability vectors with $m$-components.

Given $\underline p= (p_1,\ldots, p_m)\in \mathrm{int}(\Delta^{m-1})$ the product measure
$\mathbf{p}:=(p_1,\ldots, p_m)^\N$ is called a \textit{Bernoulli measure} in the space of sequences
$X:=\{1,\ldots, m\}^\N$. Endowed with this measure the one-sided shift $\sigma:X\to X$,
is called a \textit{Bernoulli shift}. The measure preserving dynamical system $(X,\sigma,\mathbf{p} )$
is ergodic and mixing and is fixed from now on.

A \textit{non-invertible random cocycle} is a linear cocycle over a Bernoulli shift $(X,\sigma,\mathbf{p})$ determined by a locally constant function $\bld A:X\to\Mat(d,\R)$, where locally constant means that for some list   $(A_1,\ldots, A_m)\in \Mat(d,\R)^m$,
$\bld A(\omega):=A_{\omega_0}$ for all $\omega\in X$.
Hence, a random cocycle is determined by a list $  \underline A= (A_1,\ldots, A_m)$ in the set
$$   \mathscr{C}(d):=  \Mat(d,\R)^m .$$
We denote by $F=F_{\underline A}$ the associated skew-product transformation.
In some formulations of the theory it is also convenient to represent the random cocycle $\underline A$ as the following   finitely  supported probability measure 
$$\mu_{\underline A}:=\sum_{j=1}^m p_j\, \delta_{A_j} \in \Prob_c(\Mat(d,\R)).$$
We refer to all the objects  $\underline A$, $\bld A$, $\mu_{\underline A}$  and $F=F_{\underline A}$ as the random cocycle.
The space $\mathscr{C}(d)$  is a manifold of dimension $ m\, d^2$ which we refer to
as a \textit{space of random cocycles}.

Let $\Gamma(\underline A)$ denote  the semigroup generated by the support of $\mu_{\underline A}$, i.e., generated by the matrices $A_1,\ldots, A_m$.
\begin{definition}
\label{def irreducible}
We say that a random (non invertible) cocycle $\underline A$ is  \textit{reducible} if there exist proper subspaces
$\{0\}\neq V_i\subsetneq \mathrm{Range}(A_i)$ such that 
$A_i\, V_j=V_i$ for all $i,j=1,\ldots, m$.
Otherwise  $\underline A$ is called \textit{irreducible}.
\end{definition}

The  property of being irreducible  is generic, open and dense, in $\mathcal{C}(d)$.
 Given $1\leq k\leq d$, the cocycle
 $\underline A$ is called $k$-irreducible, resp. $k$-quasi-irreducible, if the random cocycle  $\wedge_k \bld A$ is irreducible, resp. quasi-irreducible.
 
 For each $A\in\Mat(d,\R)$, let $\mathrm{K}(A)$ be the kernel of $A$ and
 $\hat {\mathrm{K}} (A)\subset \Pp(\R^d)$ be the associated projective subspace.
 A probability measure $\eta\in\Prob(\Pp(\R^d))$ is called $\underline A$-stationary if
 $\eta(\hat{\mathrm{K}} (A_j))=0$, for $j=1,\ldots, m$, and 
 $$ \eta= \sum_{j=1}^m p_j\, (A_j)_\ast\eta .$$
Notice that the existence of stationary measures is not automatic as in the case of invertible random cocycles.


The rank of a matrix $A\in\Mat(d,\R)$, denoted by
$\rank(A)$, is the dimension of the linear span of its columns.

\begin{definition}
	We say that the random cocycle $\underline A$ has constant rank $k$ if
	$\rank(A_j)=k$ for all $j=1,\ldots, m$. We denote by
	$\mathcal{C}_k(d)$ the space random cocycles with
	constant rank $k$.
\end{definition}


\subsection{Large deviations estimates}
\label{Section: LDT estimates}
We say that a linear cocycle $(T,A)$ satisfies a \textit{large deviations estimate of exponential type} if 
there exist constants $\varepsilon_0>0$, $C<\infty$ and $c>0$ such that for all $0<\varepsilon<\varepsilon_0$ and $n\in\N$,
\begin{equation}
\label{LDT}
\mu\left\{ x\in X\colon \, \left|\frac{1}{n}\, \log \norm{A^n(x)}  -L_1(T,A) \right| > \varepsilon \, \right\} \leq C\, e^{-n\, c\,\varepsilon^2\,  } .
\end{equation} 

When the cocycle $(T,A)$ varies in some space of linear cocycles we say that it satisfies
a \textit{uniform large deviations estimate of exponential type}  if there exists a neighbourhood of the original cocycle and positive constants constants
$\varepsilon_0$, $C, c$ such that 
~\eqref{LDT} holds for all $n\in\N$, $0<\varepsilon<\varepsilon_0$ and every cocycle
in the said neighbourhood.

\subsection{Markov shift}
Let $\Sigma$ be a compact metric space (of symbols) with Borel $\sigma$-algebra $\mathcal{B}$
and $X:=\Sigma^\N$ be the corresponding space of sequences.

 A \textit{Markov kernel} on $\Sigma$  is a map $P\colon \Sigma\times \mathcal{B}\to[0,1]$ such that 

\begin{itemize}
\item [(a)] for any $x\in \Sigma$, the map $E\mapsto P(x,E)$ is a probability measure in $(\Sigma,\mathcal{B})$, 
\item[(b)] for every measurable set $E\in\mathcal{B}$ the function $x\mapsto P(x,E)$ is $\mathcal{B}$-measurable.
\end{itemize}

In particular when $\Sigma$ is finite the kernel $P$ is fully characterized by the row stochastic matrix $P$ with  entries  $P_{i,j}:=P(i, \{j\})$. 

The iterated kernels are defined recursively as 
\begin{enumerate}
\item $P^1=P$, 
\item $P^{n+1}(x,E)=\int P^n(y,E)P(x,dy)$ for every $n\geq1$.
\end{enumerate}

A probability measure $\nu$ on $(\Sigma,\mathcal{B})$ is called $P$-stationary if
$$\mu(E)=\int P(x,E)\mu(dx).$$

In this case we say that   $(P,\mu)$ is \emph{strongly mixing} if there exist constants $C>0$ and $0<\rho<1$ such that for every $f\in L^\infty(\Sigma)$,
$$\label{sm}|\int f(y)P^n(x,dy)-\int f(y)\mu(dy)\|\leq C \rho^n\|f\|_{\infty}$$
is satisfied for $x\in \Sigma$ and $n\in \N$. 
If $\Sigma$ is finite and the stochastic
matrix $P$ is primitive, i.e., there exists $n\geq 1$ such that for all $i,j=1,\ldots, n$, $P_{i,j}^n>0$, then $P$ admits a unique stationary measure $\mu$ such that $(P,\mu)$ is strongly mixing.

Let $X$ be the space of sequences  $X:=\Sigma^\N$ with its product $\sigma$-algebra $\mathcal{P}$.
As before we denote by $\sigma:X\to X$ the one sided shift,
$\sigma(\omega_j)_{j\in\N}:= (\omega_{j+1})_{j\in\N}$.
The pair $(P,\mu)$ determines a unique shift invariant probability measure $\mathbb{P}_{\mu}$  on $X$, called the Kolmogorov extension of $(P,\mu)$,
see~\cite[Definition 5.3]{DK-book}.
When  $\Sigma$ is finite  the Kolmogorov measure of the  cylinder
$$ [k; i_0, \ldots, i_{n-1}]:=\left\{ \omega=(\omega_j)_{j\in\N}\in X  \colon \omega_k=i_0, 
 \,\ldots ,\, \omega_{k+n-1}=i_{n-1} \right\} $$
 is given by
 $$ \mathbb{P}_\mu( [k; i_0, \ldots, i_{n-1}]) =  
 \mu\{i_0\}\, P_{i_0, i_1}\, \cdots\, P_{i_{n-2}, i_{n-1}} . $$
 
 The measure preserving dynamical system $(\sigma, \mathbb{P}_\mu)$
 is called a Markov shift and it is mixing when $(P,\mu)$ is strongly mixing, see~\cite[Proposition 5.1]{DK-book}.

Let 
$\mathscr{C}^\infty_2(X)$ be the space of continuous functions 
$A:X\to \GLmR$ such that $A(\omega)$ depends only on the two
coordinates $\omega_0$ and $\omega_1$.
Such cocycles, over $(\sigma, \mathbb{P}_\mu)$, are determined by 
continuous functions $A:\Sigma\times\Sigma\to\GLmR$.
If $\Sigma$ is finite $A\in \mathscr{C}^\infty_2(X)$ is determined by a finite family of transition matrices $A_{i,j}$ indexed in $i,j\in\Sigma$.

We say that the linear cocycle $(\sigma, A)$, with  $A\in \mathscr{C}^\infty_2(X)$, is irreducible if there exists no measurable function $V:\Sigma\to\Gr(\R^m)$ such that
$A(x,y)\,V(x)=V(y)$ for all $x,y\in\Sigma$, where $\Gr(\R^m)$ denotes the Grassmannian manifold of all linear subspaces of $\R^m$.

\begin{proposition}[Theorem 5.3 of~\cite{DK-book}] \label{18}
Given a Markov kernel $P$ on the compact metric space $\Sigma$ and let $\mu \in\text{Prob}(\Sigma)$ a $P$-stationary measure, if we assume
\begin{enumerate}
	\item $(P, \mu)$ is strongly mixing ($P$ primitive when $\Sigma$ is finite);
	\item $L_1(\sigma,  A)>L_{2}(\sigma,  A)$;
	\item $(\sigma,  A)$ is irreducible.
\end{enumerate}	
Then there exists a neighborhood $V$ of $A$ in $\mathscr{C}^\infty_2(X)$ and there exist $C<\infty, \ k>0$ and $\varepsilon_0>0$ such that for all $0<\varepsilon<\varepsilon_0$, $n\in\N$  and $B\in V$ 

$$\mu\left\lbrace |\frac{1}{n}\log\|B^n\|-L_1(\sigma,B)|>\varepsilon\right\rbrace\leq Ce^{-k\varepsilon^2n}.$$
\end{proposition}

\bigskip

\subsection{Main results}

We now state the main results of this article.

We say that $\underline A=(A_1,\ldots, A_m)\in \Mat(d,\R)^m$ has constant rank $k$ if $\rank(A_j)=k$ for all $1\leq j\leq m$.
Let
$$ \mathcal{C}_k(d):=\left\{\underline A\in \mathscr{C}(d) \, :\, \underline A\, \text{ has constant rank } \, k \,  \right\} . $$ 
\begin{proposition}
\label{prop 1}
If  $\underline A\in\mathcal{C}_k(d)$   then $L_{k+1}\underline A=-\infty$.\\
Moreover,\,   $\rank(A_i\, A_j)=k$ for every $i,j=1,\ldots, m$ \, $\Leftrightarrow$ \, $L_k(\underline A)>-\infty$,\; 
and the set of cocycles $A\in \mathscr{C}_k(d)$ where this property holds is open and dense.
\end{proposition}

\begin{theorem}
\label{main theorem}
Given $\underline A\in\mathscr{C}(d)$ and $1\leq i \leq k \leq d$, assume
\begin{enumerate}
	\item $\underline A$ has constant rank $k$;
	\item $\rank(A_i\, A_j)=k$ for every $i,j=1,\ldots, m$;
	\item $L_i(\underline A)>L_{i+1}(\underline A)$;
	\item $\wedge_i \underline A$ is  irreducible.
\end{enumerate}	
Then
\begin{enumerate}
	\item[(a)] all cocycles $\wedge_i \underline B$ with $\underline B\in\mathcal{C}_k(d)$ in a small neighborhood of
	$\underline A$ satisfy uniform large deviations estimate of exponential type;
	\item[(b)] the function $\mathcal{C}_k(d)\ni \underline B \mapsto L_1( \underline B) + \cdots + L_i( \underline B)$ is H\"older continuous in a small  neighborhood of
	$\underline A$.
\end{enumerate}
\end{theorem}

\begin{corollary}
\label{coro 1}
Given $\underline A\in\mathscr{C}(d)$ and $1\leq i \leq k \leq d$, assume
\begin{enumerate}
	\item $\underline A$ has constant rank $k$;
	\item $\rank(A_i\, A_j)=k$ for every $i,j=1,\ldots, m$;
	\item $L_{i-1}(\underline A)>L_i(\underline A)>L_{i+1}(\underline A)$;
	\item  $\wedge_{i-1} \underline A$  and  $\wedge_{i} \underline A$   are irreducible. 
\end{enumerate}	
Then  the function $\mathcal{C}_k(d)\ni  \underline B \mapsto  L_i( \underline B)$ is H\"older continuous in a small  neighborhood of
	$\underline A$.
\end{corollary}

\begin{corollary}
\label{coro 2}
	Given $\underline A\in\mathscr{C}(d)$, assume
	\begin{enumerate}
		\item $\underline A$ has constant rank $k$;
		\item $\rank(A_i\, A_j)=k$ for every $i,j=1,\ldots, m$;
		\item $\underline A$ has simple Lyapunov spectrum;
		\item  $ \wedge_i \underline A$ is  irreducible for all $1\leq i\leq k$. 
	\end{enumerate}	
	Then  the functions $\mathcal{C}_k(d)\ni \underline B\mapsto  L_i(\underline B)$, with $1\leq i\leq k$,  are all  H\"older continuous in a small  neighborhood of
	$\underline A$.
\end{corollary}

\begin{corollary}
	\label{coro 3}
	Given $\underline A\in\mathscr{C}_k(d)$, if
	\begin{enumerate}
		\item $\rank(A_i\, A_j)=k$ for every $i,j=1,\ldots, m$;
		\item $\wedge_i \underline A$ is  irreducible for all $1\leq i\leq k$; 
	\end{enumerate}	
	then  the functions $\mathcal{C}_k(d)\ni \underline B \mapsto  L_i(\underline B)$, with $1\leq i\leq k$,  are all  continuous at
	$\underline A$.
\end{corollary}

\begin{remark}
The full characterization of the continuity of 
$\mathscr{C}_k(d)\ni \underline A\mapsto L_i(\underline A)$ in the renaming reducible cases depends on the result announced by Avila, Eskin and Viana, referred to in the introduction, and its generalization to random cocycles over mixing Markov shifts.	
\end{remark}

The next and final proposition gives a simple criterion to check assumption (4) in Theorem~\ref{main theorem} and corollaries~\ref{coro 1} and~\ref{coro 2}

\begin{proposition}
\label{Zariski dense criterion}
Given $\underline A\in\mathscr{C}(d)$, assume that
\begin{enumerate}
	\item $\underline A$ has constant rank $k$;
	\item $\rank(A_i\, A_j)=k$ for every $i,j=1,\ldots, m$.
\end{enumerate}	
For any   $1\leq l \leq m$, if $R_l:=\mathrm{Range}(A_l)$ then
$A_l\, A_{i_1}\, \cdots \, A_{i_n}\vert_{R_i}:R_l\to R_l$
is an isomorphism for every $i_1, \ldots, i_n\in \{1,\ldots, m\}$, $n\in\N$.
Denoting by  $\mathcal{G}_l$ the group generated by these automorphisms of $R_l$,
we have:
\begin{enumerate}
	\item[(a)] The groups $\mathcal{G}_i$ and $\mathcal{G}_j$ are conjugated for all $1\leq i,j\leq m$;
	\item[(b)] If some $\mathcal{G}_l$ is Zariski dense then $\wedge_i  \underline A$ is  irreducible for all $1\leq i\leq k$. 
\end{enumerate}
\end{proposition}

\section{An example}
\label{example}

In this section we provide a simple example of a family of random non-invertible cocycles with non constant rank where the Lyapunov exponent is nowhere continuous.

Consider the matrices in $\Mat_2(\R)$,
$$ P=\begin{bmatrix}
	1 & 0 \\0 & 0
\end{bmatrix}\quad  \text{ and } \quad
R_\alpha =\begin{bmatrix}
	\cos\alpha  & -\sin\alpha \\ \sin\alpha & \cos\alpha
\end{bmatrix} $$
and the family of linear cocycles $F_\alpha \equiv\left( (\frac{1}{2},\frac{1}{2}),\, (P, R_\alpha) \right)$.

\begin{proposition}
\label{prop 2}
The following hold:
	\begin{enumerate}
		\item If $\alpha\in \pi/2\,\Q$ there exists $n\in\N$ such that
		$P\, R_\alpha^n\, P=0$ and 
		 $L_1(F_\alpha)=-\infty$.
		\item If $\alpha\notin 2\pi\,\Q$ then $F_\alpha$ admits the following stationary measure
		$$ \nu= \sum_{j=0}^\infty \frac{1}{2^{j+1}}\, \delta_{\widehat{ R_\alpha^j (1,0)}} .$$
		\item $\displaystyle  L_1(F_\alpha)=\sum_{j=0}^\infty \frac{1}{2^{j+1}}\,\log \abs{\cos(j\,\alpha)} $.
		\item 
		The set $\{ \alpha\in\R \colon L_1(F_\alpha)>-\infty\}$ has full Lebesgue measure.
		\item The function
		$\alpha\mapsto L_1(\mu_\alpha)$ is nowhere continuous.
	\end{enumerate}
\end{proposition}

\begin{proof}
For proving $(1)$ we have that the composition $P R^n_{\alpha}P$ is given by $$ P R^n_{\alpha}P=\begin{bmatrix}
	\cos (n\alpha) & 0 \\ 0 & 0
\end{bmatrix}\quad,$$ 
consequently since there exists a multiple of $\alpha$ that $n\alpha=\pi/2$, $P R^n_{\alpha}P=0$. Then if we consider $x\in X$ inside of the cylinder of length $n+2$ such that $A^{n+2}(x)=P R^n_{\alpha}P$, the Lyapunov exponent is $L_1(F_\alpha,x)=-\infty$ inside a positive measure cylinder, therefore $L_1(F_{\alpha})=-\infty$.

Next we verify $(2)$. Since the kernel $ K(P)$ of $P$ is the line spanned by $(0,1)$ and $R_{\alpha}$ is invertible, we must only verify that $\nu( \hat{K}(P))=0$ and $\nu=(P_*\nu+{R_{\alpha}}_*\nu)/2$.

Clearly since ${\alpha}\not\in 2\pi\Q$, in particular $j\alpha\neq \pi/2 +2k\pi$ for every $j$, then $\widehat{ R_\alpha^j (1,0)}\neq\hat K(P)$ and $\nu (\hat K(P))=0$. 

On the other hand, for proving the invariance we take a continuous function $\varphi\colon\mathbb{P}(\R^2)\to \R$ and set

$$\begin{aligned}
\frac{1}{2}\int_{\mathbb{P}(\R^2)}\varphi(x) d(P+R_{\alpha})_*\nu=&\frac{1}{2}\int \varphi(P (x))+ \varphi(R_{\alpha}(x))d\nu\\
=&\sum_{j=0}^{\infty}\frac{1}{2^{j+2}}\varphi(\widehat{P\circ R^j_{\alpha}(1,0)})+ \varphi(\widehat{R^{j+1}_{\alpha}(1,0)}).\\
\end{aligned}.$$
In particular for every $j$ we have $\widehat{P\circ R^j_{\alpha}(1,0)}=\widehat{(\cos(j\alpha),0)}=\widehat{(1,0)}$, therefore $\varphi(\widehat{P\circ R^j_{\alpha}(1,0)})=\varphi\widehat{(1,0)}$. Observing that $\varphi\widehat{(1,0)}=\varphi(\widehat{R_{\alpha}^0(1,0)})$, we compute

$$\begin{aligned}
\frac{1}{2}\int_{\mathbb{P}(\R^2)}\varphi(x) d(P+R_{\alpha})_*\nu =&\frac{1}{2}\varphi\widehat{(1,0)}+\sum_{j=0}^{\infty}\frac{1}{2^{j+2}}\varphi(\widehat{R^{j+1}_{\alpha}(1,0)})\\
=& \frac{1}{2}\varphi(\widehat{R_{\alpha}^0(1,0)})+\sum_{j=1}^{\infty}\frac{1}{2^{j+1}}\varphi(\widehat{R^{j}_{\alpha}(1,0)})\\
=& \int \varphi(x)d\nu.\\
\end{aligned}$$
This concludes that $\nu$ is in fact an stationary measure.

In $(3)$ we need to see that $\displaystyle  L_1(F_\alpha)=\sum_{j=0}^\infty \frac{1}{2^{j+1}}\,\log \abs{\cos(j\,\alpha)} $. For this we consider the cylinder $C=[P]$ which can be seen as the infinite union of cylinders $C_j=[PR_{\alpha}^jP]$ for $j\geq 0$ with measure $1/2^{j+2}$. Let $g\colon C\to C$ be the first return map to $C$ given by $g(x)=\sigma^{j(x)}(x)$ where $j(x)=\min\{i\in\N\, ,\, \ \sigma^{i}(x)\in C\}$ and define the cocycle $F_C\colon X\to \Mat_2(\R)$ over $g$ by $\textbf{C}(x)=PR^{j(x)}_{\alpha}P$ with $\mu\vert_C$ the normalized product measure. 

Since $\textbf{C}$ is a rank $1$ cocycle, the logarithm of its norm becomes an additive process and 

$$\begin{aligned}L_1(F_C)=&\lim_{k\to \infty} \frac{1}{k}\log \|\textbf{C}^k(x)\| \\
=&\lim_{k\to \infty} \frac{1}{k}\sum_{j=1}^{k-1}\log \|\textbf{C}(g^j x)\|
=&\int \log \|\textbf{C}(y)\| d\mu_C\\
=& \sum_{j=0}^\infty \frac{1}{2^{j+1}}\log \vert\cos(j\alpha)\vert.
\end{aligned}$$

Finally the relation between the Lyapunov exponent $L_1(F_C)$ and $L_1(F_\alpha)$ is given by $L_1(F_C)=L_1(F_C,x)=c(x)L_1(F_\alpha,x)$ where $c$ is invariant by $\sigma$, see Viana \cite[Proposition 4.18 and Exercise 4.8]{LLE}, which means that is constant in a full measure set and in fact
$$c=\lim \frac{1}{k}\sum_{i=0}^{k-1} j(g^i(x))=\int j(x)d\mu_C=\sum_{i=1}^\infty\frac{j}{2^{j+1}}$$
which equals to $1$, concluding that $L_1(F_\alpha)=L_1(F_C)$.

In $(4)$ we need to prove that the set $\{\alpha\in \R\, : \, L_1(F_{\alpha})>-\infty\}$ has full Lebesgue measure, thus, it is enough to see that the measure of
$$ S_j=\left\lbrace \alpha\in\R : \, |\log|\cos(\alpha j)||>\frac{2^{j+1}}{j^2}\right\rbrace$$
converges to zero when $j$ goes to infinity. This is a consequence of the fact that if $\alpha\in S_j$ then $|\alpha j-\frac{\pi}{2}|< e^{- \frac{2^{j+1}}{j^2}}$ and the sum $\sum_{j=1}^\infty |S_j|<\infty$. Therefore, the limit set $\cap_{k=0}^\infty\cup_{j=k}^\infty S_j$ has measure zero, whose complement is the set of $\alpha\in \R$ satisfying $L_1(F_{\alpha}>-\infty$, which completes $(4)$. 

The item $(5)$ is a clear consequence of $(1)$ and  $(4)$. 

\end{proof}

\begin{remark}
The previous example can be extended to higher dimensions. Take an orthogonal projection $P$ onto some subspace 
$E\in\Gr_k(\R^d)$ and a $1$-parameter group of rotations $R_t$ such that $R_0=R_1=I$ and $R_t E= E^\perp$ for some $t\in\R$. Then $F_t\equiv \left( ( \frac{1}{2}, \frac{1}{2}),\, (P,R_t)  \right)$ satisfies similar properties as above.
\end{remark}

\section{Reduction to invertible cocycles}
\label{reduction}

In this section we prove that any cocycle satisfying hypothesis (1) and (2) of Theorem~\ref{main theorem} is semi-conjugated to a
locally constant invertible cocycle  (over a Bernoulli shift) with values in $\GL(k,\R)$. Locally constant here means that the cocycle depends on two consecutive symbols.

\begin{lemma}
	Given $A,B\in\Mat(d,\R)$  with $\rank(A)=\rank(B)=k$, the following are equivalent:
	\begin{enumerate}
		\item $A\,B$ has rank $<k$;
		\item $\wedge_k(A B)=0$.
	\end{enumerate}
\end{lemma}

\begin{proof}
Clearly, if we assume that $\text{rank}AB<k$, then any combination of $k$ vectors $v_1,\ldots, v_k$ generates a sequence linearly dependent $ABv_1, \ldots ABv_k$. Therefore $\wedge_kAB(v_1\wedge\ldots\wedge v_k)=0$. 
The converse is immediate.
\end{proof} 

Given $\underline A\in \mathcal{C}_k(d)$ with constant rank $k$ define 
$$\Theta_k(\underline A):= \min\left\{ \norm{\wedge_k(A_i\, A_j)} \colon 1\leq i,j\leq m \right\} .$$

\begin{proposition}\label{42}
Given $\underline A\in \mathcal{C}_k(d)$, then $ \Theta_k(\underline A)>0$ if and only if  
 $L_k (\underline A)>-\infty$.
\end{proposition}

\begin{proof}
Assume first that $\Theta_k(\underline A)=0$, then there exists a positive measure cylinder $C$ such that $\|\wedge_k \bld A^n\|=0$ for every $\underline{A}\in C$. Therefore, if we compute the Lyapunov exponent we get that $L_k(\underline A)=-\infty$.

Conversely, if  $\Theta_k(\underline{A})=a>0$ then $\|\wedge_k  (A_iA_j)v\|\geq a \|v\|$ for every $v\in \wedge_k\R^m$.
We claim that $\|\wedge_k  \bld A^nv\|\geq a^{n/2} \|v\|$ for every $n\in\N$. 
This is clear for $n=2$.
Assuming  this holds for every $m<n$, and  for some appropriate symbols $1\leq i,j\leq m$,
$$\|\wedge_k  \bld A^{n+1}v\|=\|\wedge_k A_jA_i \bld A^{n-2}v\|\geq a \|\wedge_k \bld A^{n-2}v\|\geq a^{n/2} \|v\|.$$

Therefore $L_1(\wedge_k \bld A)\geq a$ for every $\underline{A}\in \mathcal{C}_k(d)$, and since $L_1(\wedge_k \bld A) = L_1(\underline{A}) +\ldots+L_k(\underline{A})$
it follows that $L_k( \underline{A})>-\infty$.
\end{proof}

\begin{proposition}\label{43}
The following set  
$$ \Uscr_k:=\left\{  \underline A \in \mathcal{C}_k(d) \colon  \, \Theta_k(\underline A)>0 \, \right\}   $$
is open and dense in $\mathcal{C}_k(d)$.
\end{proposition}

\begin{proof}
	Follows from the continuity of the map
	$\mathcal{C}_k(d)\ni \underline A\mapsto \Theta_k(\underline A)$.
	The density of $\Uscr_k$ in  $\mathcal{C}_k(d)$ is clear.
%
%
\end{proof}

\bigskip

Take $\underline A \in\Uscr_k$, which will be fixed throughout this section.
Given $1\leq i\leq m$, let 
$R_i:=\mathrm{Range}(A_i)\in \Gr_k(\R^d)$.
Consider a family of $d\times k$ rectangular matrices
$\{ M_i\}_{1\leq i\leq m}$ 
such that for  $i=1,\ldots, m$,
\begin{enumerate}
	\item $M_i$ is orthogonal, i.e., $M_i^t M_i=I$;
	\item 	$\mathrm{Range}(M_i)=R_i$.
\end{enumerate}
As a consequence,
\begin{equation}
	\R^d=\mathrm{K}(A_i)\oplus  R_j,\quad \forall \; 1\leq i,j\leq m .
\end{equation}
Define, for $1\leq i,j \leq m$,
\begin{equation}
\label{Cij}
C_{i,j}:= M_i^t \, A_i \, M_j . 
\end{equation}  
\begin{lemma}
	Given $1\leq i,j\leq m$,  
	\begin{enumerate}
		\item $C_{i,j} \in \GL(k,\R)$,
		\item $M_i\, C_{i,j} = A_i \, M_j$.
	\end{enumerate}
	
\end{lemma}

\begin{proof}
	Since $A_i$ has rank $k$ and $M_j$ is injective, $A_i M_j$ has rank $k$.
	Because $M_i^t M_i=I$, the matrix  $M_i M_i^t$ represents the orthogonal projection onto $R_i$. Hence
	$M_i\, C_{i,j}=(M_i M_i^t) A_i M_j=A_i M_j$ has rank $k$, which implies that
	$C_{i,j}$ has also rank $k$, i.e., $C_{i,j}\in \GL(k,\R)$.
\end{proof}

%

Consider the invertible random cocycle
$\tilde F:X\times\R^k\to X\times \R^k$ defined by
$\tilde F(\omega, v):=(\sigma\omega, \bld C(\omega)\, v)$, where $\bld C(\omega):=C_{\omega_1 \, \omega_0}$.
Although the base map is the same Bernoulli shift, because the matrix valued function $\bld C$ depends now on two consecutive coordinates we will regard it as  a linear cocycle over a Markov shift.

Consider also the embedding 
$H:X\times\R^k\to X\times \R^d$ defined by
$H(\omega, v):=(\sigma\omega, M_{\omega_0} v)$.

\begin{proposition}\label{45}
	Given $(\underline p, \underline A)\in \Uscr_k$,
	the random Markov cocycle $\tilde F$ is semi-conjugated to $F$ in the sense that 
$$\begin{CD}
	X\times\R^k  @>\tilde F >> X\times\R^k\\
	@VHVV @VVHV\\
	X\times\R^m  @>>F>  X\times\R^m
\end{CD}    $$
is a commutative diagram.	Moreover $L_i(F)=L_i(\tilde F)$
for $i=1,\ldots, k$.
\end{proposition}

\begin{proof}
	
The relation $F \circ H=H\circ\tilde{F}$ holds because
$A_{\omega_1}M_{\omega_0}=M_{\omega_1}C_{\omega_1\omega_0}$.
In fact 
\begin{align*}
	(F\circ H)(\omega,v)=F(\sigma\omega,M_{\omega_0}v)&=(\sigma^2\omega, A_{\omega_1}M_{\omega_0})\\
	&=(\sigma^2\omega,M_{\omega_1}C_{\omega_1\omega_0})=(H\circ\tilde F)(\omega,v) .
\end{align*}
This also implies that
\begin{equation}
	\label{semi-conjugation C-A}
	M_{\omega_n}\, \bld C^n(\omega) = \bld A^n(\sigma\omega)\, M_{\omega_0}, \quad \forall \omega\in X.
\end{equation}
	
	Oseledets theorem holds for fiber non-invertible cocycles
	such as $F$. See for instance~\cite[Theorem 44]{DK-book},  where the assumption that the base dynamics is invertible is not an issue here because we can replace the one-sided shift $\sigma$ by its natural extension, the two-sided shift.
	
	On the other hand the Oseledets theorem also holds for the 
	 fiber  invertible cocycle  $\tilde F$ and its integrability follows trivially from the fact that its action on the fibers is given by finitely many invertible matrices $C_{i,j}$. Thus all Lyapunov exponents of $\tilde F$ are finite.

	The semi-conjugation $H$ is not injective
	but it is still fiber-wise injective with
	$H_\omega\R^k=M_{\omega_0}\R^k=R_{\omega_0}$ being a $k$-dimensional subspace. Moreover,
	$\bld A^n(\sigma\omega)\, R_{\omega_0} = R_{\omega_n}$,
	and by equation ~\eqref{semi-conjugation C-A}, the action of the linear map
	$\bld A^n(\sigma\omega)\vert_{ R_{\omega_0} }: R_{\omega_0}\to  R_{\omega_n}$ is conjugated to that of 	
	$\bld C^n(\omega)$ on $\R^k$. Hence with probability one,
	any Lyapunov exponent of $\bld A^n(\sigma\omega)$ along a vector in $R_{\omega_0}$ equals one of the (finite)  Lyapunov exponents of 
	$\bld C^n(\omega)$.  On the other hand, the Lyapunov exponent
	of $\bld A^n(\sigma\omega)$ along any vector in $\Ker(A_{\omega_1})$ is $-\infty$. Hence this subspace is part of the Oseledets' s filtration of the cocycle $F$ at the base point $\sigma\omega$. Finally, because    $\R^d= \Ker(A_{\omega_1})\oplus R_{\omega_0}$ it follows that the cocycles $F$ and $\tilde F$ have the same first $k$ Lyapunov exponents, all of them being finite.
\end{proof}

\section{Proof of the main results}
\label{proof}

\begin{proof}[Proof of Proposition~\ref{prop 1}]
For the first part take any combination of $k+1$ vectors $v_1,\ldots,v_{k+1}$, since $\rank(A_j)=k$,
 $$\wedge_{k+1} A_j (\omega)\, v_1\wedge\ldots \wedge v_{k+1}=0.$$ In particular $\wedge_{k+1} \bld A^n(x) \, v_1\wedge\ldots \wedge v_{k+1}=0$ for every $n\in\N$, which implies that $\|\wedge_{k+1} \bld A^n (\omega)\|=0$ and $L_1(\bld A) + \cdots + L_{k+1}(\bld A) = L_1( \wedge_{k+1}\bld A)=-\infty$. 	
Hence $L_{k+1}(  \underline  A)= L_{k+1}( \bld A)=-\infty$.

The second statement follows from Proposition~\ref{42}.

Finally, the third statement stems from Proposition~\ref{43}.
\end{proof}


\begin{proof}[Proof of Theorem~\ref{main theorem}]
Taking exterior powers we reduce the proof to the 
case $i=1$. 

To prove item (a) we  use Proposition \ref{45} to semi-conjugate our original cocycle to a random cocycle over a Markov shift and then  apply Proposition~\ref{18}. 

Notice that the matrices $M_i$ introduced in the previous section can be selected   depending continuously on $\underline A$ in a neighborhood of this cocycle. 
Hence the map $\underline A\mapsto C_{i,j}$, defined by~\eqref{Cij}
is also continuous in that same neighborhood. In other words, we can make the Markovian cocycle $\bld C$ to depend continuously on $(\underline p, \underline A)\in \mathcal{C}_k(d)$.

Next we need to check that the cocycle $(\sigma, \bld C)$ satisfies the assumptions (1)-(3) of Proposition~\ref{18}:

Item (1) holds because $\sigma:X\to X$ is the same Bernoulli shift.
The row stochastic matrix $P$ has all rows equal to $\underline p$
and is hence a primitive matrix.

Item (2) is just a reformulation of hypothesis (2) in Theorem~\ref{main theorem}.

To prove item (3) we need to see that $(\sigma, \bld C)$ is irreducible. Assuming by contradiction that  $(\sigma, \bld C)$ is reducible there exists a family of proper subspaces
$W_i\in \Gr(\R^k)$, indexed in $1\leq i\leq m$, such that
$W_i = C_{i,j}\, W_j$ for all $1\leq i,j\leq m$.
Consider the family of sub-spaces $V_i:=  M_i\, W_i\in\Gr(\R^m)$.
Then $$V_i= M_i\, W_i =M_i\, C_{i,j}\, W_j= A_{i}\, M_j\, W_j= A_i\, V_j$$
which shows that $(\underline p, \underline A)$ is reducible according to Definition~\ref{def irreducible}. Hence by hypothesis (4) of Theorem~\ref{main theorem}, the cocycle  $(\sigma, \bld C)$ must be irreducible.

From~\eqref{semi-conjugation C-A} it follows that there exists a positive and finite constant $c$ such that
\begin{equation}
\label{bounds on norm ratios}
 c^{-1}\leq \frac{ \norm{\bld A^n(\sigma\omega)}}{\norm{\bld C^n(\omega)}} \leq c \qquad \forall \omega\in X. 
\end{equation}

In fact multiplying~\eqref{semi-conjugation C-A} on the left by $M_{\omega_n}^t$, since $M_{\omega_n}^t\, M_{\omega_n} =I$ we get that
$\norm{\bld C^n(\omega)}\lesssim \norm{\bld A^n(\sigma\omega)}$.
Conversely,  multiplying~\eqref{semi-conjugation C-A} on the right by $M_{\omega_0}^t$, since  $M_{\omega_0}\, M_{\omega_0}^t$  is the orthogonal projection on $R_{\omega_0}$, we get that
\begin{equation*}
\norm{\bld A^n(\sigma \omega)\vert_{R_{\omega_0}}}\lesssim \norm{\bld A^n(\sigma \omega)\, M_{\omega_0}\, M_{\omega_0}^t}\lesssim \norm{\bld C^n( \omega)}.
\end{equation*} 
Because $\Theta_k(\underline A)>0$, \, $\Ker(A_{\omega_1})\oplus R_{\omega_0}=\R^d$ and the restriction of the orthogonal projection 
from $R_{\omega_0}$ onto $\Ker(A_{\omega_1})^\perp$ is an isomorphism 
with positive co-norm $\kappa>0$. Therefore
$$\norm{\bld A^n(\sigma \omega)} = \norm{\bld A^n(\sigma \omega)\vert_{\Ker(A_{\omega_1})^\perp }} \leq \kappa^{-1}\, \norm{\bld A^n(\sigma \omega)\vert_{R_{\omega_0}}} \lesssim \norm{\bld C^n( \omega)}. $$
 
The bounds~\eqref{bounds on norm ratios} imply that
$$ \frac{1}{n}\, \log \norm{\bld A^n(\sigma\omega)} 
= \frac{1}{n}\, \log \norm{\bld C^n(\omega)} + O\left(\frac{1}{n} \right) $$
which allows us to transfer large deviation estimates from the cocycle $\bld C$ to $\bld A$. Finally the conclusion (a) follows from Proposition \ref{18} and the continuity of the
map $\underline A\mapsto \bld C$.

Conclusion (b) follows from the Abstract Continuity Theorem~[Theorem 3.1]\cite{DK-book}. The main assumption of this theorem holds  because we have (fiber) uniform LDT (Large Deviation Type) estimates in a neighborhood $\mathscr{V}$ of   $ \underline A \in \mathcal{C}_k(d)$. The exponential type of these LDT estimates leads to a H\"older modulus of continuity of the function $\mathscr{V}\ni  \underline A \mapsto L_1(\underline A)$.
\end{proof}


\begin{proof}[Proof of Corollary~\ref{coro 1}]
	
Apply Theorem~\ref{main theorem} twice, with indexes $i-1$ and $i$
and notice that $L_i(\underline A)=L_1(\wedge_i  \underline A) - L_1(\wedge_{i-1}  \underline A)$.	
\end{proof}


\begin{proof}[Proof of Corollary~\ref{coro 2}]
Apply Corollary~\ref{coro 1} to all indexes $1\leq i\leq k$.
\end{proof}


\begin{proof}[Proof of Corollary~\ref{coro 3}]
It is enough to prove that if 
\begin{equation}
\label{gap pattern}
L_{i-1}(\underline A)
< L_{i}(\underline A)=L_{i+k}( \underline A)<L_{i+k+1}(\underline A)
\end{equation}
then the maps $ \underline A\mapsto f_j(\underline A):= L_{i+j-1}( \underline A)$ are continuous at 
$\underline A$, for every $j\in \{1, \ldots, k+1\}$, 
which is a consequence of the following result, see~\cite[Lemma 6.1]{DK2}:

\begin{lemma}\label{abstract-cont-lemma}
	Given a topological space $\tops$ and $a \in \tops$   let $f_1, f_2, \ldots, f_p \colon \tops \to \R$ be functions such that:
	
	\begin{enumerate}
		\item[(a)]   $f_1(x) \ge f_2(x) \ge \ldots \ge f_p(x)$ for all  $x \in \tops$;
		
		\item[(b)]   $f_1 (a) = f_2 (a) = \ldots = f_p (a)$;
		
		\item[(c)]   $g := f_1 + f_2 + \ldots + f_p$ is continuous at $a$; 
		
		\item[(d)]   $f_1$ is upper semi-continuous at $a$.
	\end{enumerate}

	Then  for every $1 \le j \le p$, $f_j$ is continuous at $a$.
	
\end{lemma}

 By Corollary~\ref{coro 1}, the function
$$\underline B\mapsto L_i(\underline B)+\cdots + L_{i+k}(\underline B)= L_1(\wedge_{i+k} \underline B)-L_1(\wedge_{i-1} \underline B)$$
is H\"older continuous in a neighborhood $\mathscr{V}$ of $\underline A$. In particular hypothesis (c) above holds. Assumptions (a) and (b) of Lemma~\ref{abstract-cont-lemma}  hold because of the gap pattern~\eqref{gap pattern}.

By Theorem~\ref{main theorem},
the function $\underline B \mapsto L_1(\wedge_{i-1}(\underline B))=L_1(\underline B)+\cdots + L_{i-1}(\underline B)$
is H\"older continuous and in particular continuous at $\underline A$.
On the other hand, the function $\underline B \mapsto L_1(\wedge_{i}(\underline B))=L_1(\underline B)+\cdots + L_{i}(\underline B)$
is always upper semi-continuous at $\underline A$.
Hence the difference
$\underline B \mapsto   f_1(\underline B)=L_{i}(\underline B) = L_1(\wedge_{i}(\underline B)) - L_1(\wedge_{i-1}(\underline B)) $
 upper semi-continuous at $\underline A$, which proves item (d) of Lemma~\ref{abstract-cont-lemma}.
\end{proof}


\begin{proof}[Proof of Proposition~\ref{Zariski dense criterion}]
 Given $l, i_1, \ldots, i_n\in \{1,\ldots, m\}$,
 by the proof of Proposition~\ref{42},
 $\norm{\wedge_k (A_l\, A_{i_1}\, \cdots \, A_{i_n}) }>0$ which implies that $R_l$ is the range of the product matrix  $A_l\, A_{i_1}\, \cdots \, A_{i_n}$. Notice that this matrix shares its kernel with $A_{i_n}$. On the other hand by assumption (2),
 $R_l\oplus \Ker(A_{i_n})=\R^d$, which implies that
 $A_l\, A_{i_1}\, \cdots \, A_{i_n}\vert_{R_l}:R_l\to R_l$ is an isomorphism. Thus the group $\mathscr{G}_l$ generated by these automorphisms of $R_l$ is well defined.

 Given $i,j=1,\ldots, m$,  and for any  $M$ in the semigroup generated by $\{A_1,\ldots, A_m\}$, $A_j\, M$ and $A_j^{-1}$ belong to $\mathscr{G}_j$, which implies that  
 $A_j\, \mathscr{G}_i\, A_j^{-1}\subseteq \mathscr{G}_j$.
 Therefore these groups are pairwise conjugated and conclusion (a) holds.

 Assume $\mathscr{G}_l$ is Zariski dense but $\wedge_i\underline A$ is reducible according to Definition~\ref{def irreducible}.
 Then there exist proper subspaces $\{0\}\neq W_j\subset \mathrm{Range}(\wedge_i A_j)\subset \wedge_i \R^d$, such that
 $(\wedge_i A_j)\, W_n= W_j$ for all $n,j\in \{1,\ldots, m\}$.
 Therefore
 $$ \left\{ g\in \mathrm{GL}(R_l) \colon (\wedge_i g)\, W_l = W_l \, \right\}$$ 
 is a proper algebraic subgroup of $\mathrm{GL}(R_l)$   containing 
 $\mathscr{G}_l$. This contradicts the fact that $\mathscr{G}_l$ 
 is Zariski dense and concludes the proof.
\end{proof}

\medskip

\subsection*{Acknowledgments}
P.D. and C.F. were  supported by FCT-Funda\c{c}\~{a}o para a Ci\^{e}ncia e a Tecnologia through the project  PTDC/MAT-PUR/29126/2017.

C.F.  was also  supported by Conselho Nacional de Desenvolvimento Cient\'ifico e Tecnol\'ogico (CNPq). 

P.D.   was also supported by CMAFCIO through FCT project  UIDB/04561/2020.

\bigskip

\bibliographystyle{amsplain}
\bibliography{bib}

\end{document}